\documentclass[a4paper,10pt]{article}

\usepackage{cmap}
\usepackage[T1]{fontenc}
\usepackage[utf8]{inputenc}
\usepackage[english]{babel}
\usepackage{amsfonts, amssymb, amsthm}
\usepackage{xcolor}
\usepackage{amsmath}
\usepackage{a4wide}
\usepackage[affil-it]{authblk}

\usepackage{graphicx}
\usepackage{epstopdf}
\usepackage{secdot}
\usepackage{enumitem}

\numberwithin{equation}{section}

\usepackage[colorlinks=true,linktocpage,pdfpagelabels,
bookmarksnumbered,bookmarksopen]{hyperref}
\definecolor{ForestGreen}{rgb}{0.1,0.6,0.05}
\definecolor{EgyptBlue}{rgb}{0.063,0.1,0.6}
\hypersetup{
	colorlinks=true,
	linkcolor=EgyptBlue,         
	citecolor=ForestGreen,
	urlcolor=olive
}

\usepackage[hyperpageref]{backref}

\newtheorem{thm}{Theorem}[section]
\newtheorem{lemma}[thm]{Lemma}
\newtheorem{prop}[thm]{Proposition}
\newtheorem{cor}[thm]{Corollary}
\theoremstyle{definition}
\newtheorem{remark}[thm]{Remark}

\usepackage{marginnote}

\DeclareRobustCommand\nlab{\lambda}
\DeclareRobustCommand\nlabb{\lambda}

\title{
	\vspace*{-1cm}
	Estimates on the spectral interval of validity \\of the anti-maximum principle
	\\ \medskip}

\author[1]{Vladimir Bobkov\thanks{E-mail: \texttt{bobkov@kma.zcu.cz}}}
\author[1]{Pavel Dr\'abek\thanks{E-mail: \texttt{pdrabek@kma.zcu.cz}}}
\affil[1]{{\small Department of Mathematics and NTIS, 
		Faculty of Applied Sciences, 
		University of West Bohemia, 
		Univerzitn\'i 8, 301 00 
		Plze\v{n}, Czech Republic}}
\author[2]{Yavdat Ilyasov\thanks{E-mail: \texttt{ilyasov02@gmail.com}}}
\affil[2]{{\small Institute of Mathematics, 
		Ufa Scientific Center, 
		Russian Academy of Sciences, 
		112, Chernyshevsky str., 
		450008 Ufa, Russia}}

\date{}

\begin{document}
\maketitle

\begin{abstract}
	The anti-maximum principle for the homogeneous Dirichlet problem to $-\Delta_p u = \lambda |u|^{p-2}u + f(x)$ with positive $f \in L^\infty(\Omega)$ states the existence of a critical value $\lambda_f > \lambda_1$ such that any solution of this problem with $\lambda \in (\lambda_1, \lambda_f)$ is strictly negative. 
	In this paper, we give a variational upper bound for $\lambda_f$ and study its properties. 
	As an important supplementary result, we investigate the branch of ground state solutions of the considered boundary value problem in $(\lambda_1,\lambda_2)$.

	
\end{abstract}

\section{Introduction}

Consider the problem
\begin{equation}\label{D}
\tag{$\mathcal{D}_{\nlab}$}
\left\{
\begin{aligned}
-\Delta_p u &= \lambda |u|^{p-2} u + f(x) &&\text{in } \Omega,\\
u&=0 &&\text{on } \partial \Omega,
\end{aligned}
\right.
\end{equation}
where $\Delta_p u := \text{div}\left(|\nabla u|^{p-2} \nabla u\right)$, $p>1$, $\lambda \in \mathbb{R}$ is a spectral parameter, and $\Omega \subset \mathbb{R}^N$ is a bounded domain of class $C^{1,\delta}$, $\delta \in (0,1)$, with boundary $\partial \Omega$, $N \geq 1$. 
Without mentioning otherwise, we always assume that $f \in L^\infty(\Omega)$, $f \geq 0$, and $f \not\equiv 0$. 
The problem \eqref{D} is a perturbation of the nonlinear eigenvalue problem
\begin{equation}\label{E}
\left\{
\begin{aligned}
-\Delta_p u &= \lambda |u|^{p-2} u &&\text{in } \Omega,\\
u&=0 &&\text{on } \partial \Omega.
\end{aligned}
\right.
\end{equation}
Below, we will often employ the \textit{first} and \textit{second} eigenvalue of \eqref{E} which can be defined, respectively, as
$$
\lambda_1 := 
\inf\left\{
\frac{\int_\Omega |\nabla u|^p \, dx}{\int_\Omega |u|^p \, dx}:~
u \in W_0^{1,p}(\Omega) \setminus \{0\}
\right\}
$$
and
$$
\lambda_2 
:= 
\inf
\left\{
\lambda > \lambda_1:~ 
\lambda \text{ is an eigenvalue of } \eqref{E}
\right\}.
$$
Recall that $0 < \lambda_1 < \lambda_2$, \cite{anane1987}, the first eigenfunction $\varphi_1$ is unique modulo scaling, \cite{AH,Lind}, and $\varphi_1$ can be chosen strictly positive in $\Omega$, \cite{vaz}. Any second eigenfunction $\varphi_2$ has exactly two nodal domains, \cite{CDG,DR}. In particular, $\varphi_2 = \varphi_2^+ + \varphi_2^-$, where $\varphi_2^+ := \max\{\varphi_2, 0\} \not\equiv 0$ and $\varphi_2^- := \min\{\varphi_2, 0\} \not\equiv 0$. Moreover, any eigenfunction of \eqref{E}, as well as any weak solution of \eqref{D}, obeys $C^{1,\gamma}(\overline{\Omega})$-regularity for some $\gamma \in (0,1)$, \cite{anane,lieberman}.

\smallskip
Among qualitative properties of solutions of \eqref{D}, the information on a sign is of fundamental importance. 
It is well-known that the following \textit{maximum principle} is valid, \cite{FHD,vaz}: 

\begin{enumerate}[leftmargin=4em,label={$\mathcal{(MP)}$}]
	\item\label{MP}
	if $\lambda < \lambda_1$, then any solution $u$ of \eqref{D} satisfies $u>0$ in $\Omega$.
\end{enumerate}

On the other hand, it was observed by Cl\'ement \& Peletier \cite{CP} for $p=2$ and by Fleckinger et al.\ \cite{FGTT} for $p>1$ that the following \textit{anti-maximum principle} holds. Namely, there exists $\lambda_f > \lambda_1$ such that 

\begin{enumerate}[leftmargin=4em,label={$\mathcal{(AMP)}$}]
	\item\label{AMP}
	if $\lambda \in (\lambda_1, \lambda_f)$, then any solution $u$ of \eqref{D} satisfies $u < 0$ in $\Omega$.
\end{enumerate}

\noindent
We will always assume that $\lambda_f$ is the maximal value such that \ref{AMP} holds.

Although $\lambda_f>\lambda_1$ for any fixed $f$, it is known that $\lambda_f$ depends on $f$ and cannot be bounded away from $\lambda_1$ uniformly with respect to $f$, see \cite{ACG,DFG,sweers}. 
However, apart from this fact, not much is known about other properties of $\lambda_f$, and it seems that even constructive bounds for $\lambda_f$ have not been systematically studied.
We can refer only to the following lower bound obtained in \cite{FHT2014} in the case $p=2$ and $f \in L^q(\Omega)$:
$$
\lambda_f \geq \lambda_1 + \frac{K \alpha}{\left(\int_\Omega |f^\perp|^q \,dx \right)^{1/q}},
$$
where $\alpha>0$ and $f^\perp$ are defined through the $L^2(\Omega)$-orthogonal decomposition $f = \alpha \varphi_1 + f^\perp$, $q>N$ for $N \geq 2$ and $q=2$ for $N=1$, and the constant $K$ does not depend on $f$, yet $K$ is not explicitly quantified. 
In this respect, \ref{AMP} for the analogous Neumann problem
\begin{equation}
\label{N}
\left\{
\begin{aligned}
-\Delta_p u &= \mu |u|^{p-2} u + f(x) &&\text{in } \Omega,\\
\frac{\partial u}{\partial \nu} &=0 &&\text{on } \partial \Omega,
\end{aligned}
\right.
\end{equation}
is more developed. 
Note that \ref{AMP} for \eqref{N} is valid on a maximal interval $(0,\mu_f)$, \cite{ACG,CP}. It was shown in \cite{CP} that $\mu_f \geq \mu_2/4$ when $p=2$ and $N=1$, where $\mu_2$ is the second (or, equivalently, the first nonzero) eigenvalue of the Laplace operator under zero Neumann boundary conditions.
Later, it was proved in \cite{ACG} that for $p>N$, 
$$
\mu_f > 
\inf\left\{
\frac{\int_\Omega |\nabla u|^p \, dx}{\int_\Omega |u|^p \, dx}:~
u \in W^{1,p}(\Omega) \setminus \{0\}
\text{ and }
u
\text{ vanishes on some ball in }
\Omega
\right\},
$$
and no uniform lower bound is possible provided $p \leq N$.
See \cite{GGP} and \cite{TM} for generalizations.
We also refer the reader to the survey article \cite{M} for the overview of results about \ref{AMP} for \eqref{D}, \eqref{N}, and related problems.

\medskip
The aim of this paper is to provide an explicit upper bound for the maximal value $\lambda_f$ of validity of \ref{AMP} for \eqref{D}. 
Let us define the critical value
\begin{equation}\label{eq:l*}
\lambda^*_f := \inf\left\{
\frac{\int_\Omega |\nabla u|^p \, dx}{\int_\Omega |u|^p \, dx}:~
\int_\Omega f u \, dx = 0, 
~
u \in W_0^{1,p}(\Omega) \setminus \{0\}
\right\}.
\end{equation}
\begin{thm}\label{thm:1}
	Let $p>1$. Then $\lambda_f^* \in (\lambda_1, \lambda_2]$ and the following assertions hold:
	\begin{enumerate}[label={\rm(\roman*)}]
		\item\label{thm:1:2} If $\lambda_f^* < \lambda_2$, then $\lambda_f < \lambda_f^*$.
		\item\label{thm:1:1} If $p = 2$ or $N=1$, then $\lambda_f \leq \lambda_f^*$.
	\end{enumerate}
\end{thm}

We also state some properties of $\lambda_f^*$.
\begin{prop}\label{thm:3}
	Let $p>1$. The following assertions hold:
	\begin{enumerate}[label={\rm(\roman*)}]
		\item\label{thm:3:4} If there exists a second eigenfunction $\varphi_2$ such that $\int_\Omega f \varphi_2 \, dx \neq 0$, then $\lambda_f^* < \lambda_2$ and thus $\lambda_f < \lambda_f^* < \lambda_2$.
		\item\label{thm:3:2} There exists a sequence $\{f_n\} \subset L^\infty(\Omega)$ such that $\lambda^*_{f_n} \to \lambda_1$ as $n \to +\infty$.
	\end{enumerate}
\end{prop}

Several additional basic properties of $\lambda_f$ and $\lambda_f^*$ are discussed in Lemma \ref{lem:l<l<l} below.

Apparently, the information for the linear problem ($p=2$) is more accurate. This is mainly due to the presence of the Fredholm alternative, which states, in particular, that \eqref{D} is uniquely solvable provided $\lambda \in (\lambda_1,\lambda_2)$, and \eqref{D} is solvable at $\lambda=\lambda_2$ if and only if $\int_\Omega f \varphi_2 \,dx =0$ for any second eigenfunction $\varphi_2$. While a counterpart of the Fredholm alternative for the general $p$-Laplacian is relatively well-developed when $\lambda$ is in a neighbourhood of $\lambda_1$ (see, e.g., \cite{DGTU,takac} and references therein), the situation near higher eigenvalues is more complicated due to the lack of full description of the spectrum of the $p$-Laplacian. Because of that, the inequality $\lambda_f \leq \lambda_f^*$ remains an open problem if the following three assumptions are simultaneously satisfied: $p \neq 2$, $N\geq 2$, and $\lambda_f^* = \lambda_2$.

\smallskip	
The definition of $\lambda_f^*$ and the results of Theorem \ref{thm:1} are connected with 
energy properties of solutions of \eqref{D}. 
Namely, recall that (weak) solutions of \eqref{D} are in one-to-one correspondence with critical points of the \textit{energy functional} $E_\lambda \in C^1(W_0^{1,p}(\Omega),\mathbb{R})$ given by
$$
E_\lambda(u) := \frac{1}{p} H_\lambda(u)  - \int_\Omega f u \, dx,
\quad \text{where} \quad 
H_\lambda(u) := \int_\Omega |\nabla u|^p \, dx - 
\lambda \int_\Omega |u|^p \, dx.
$$
Suppose that $u \in W_0^{1,p}(\Omega)$ is a solution of \eqref{D} with $E_\lambda(u)=0$. Since $u$ is a critical point of $E_\lambda$, we have, in particular, that $\left<E_\lambda'(u),u\right> \equiv H_\lambda(u)-\int_\Omega f u \, dx = 0$. 
It follows from $E_\lambda(u) = \left<E_\lambda'(u),u\right> = 0$ that 
$$
\lambda
=
\frac{\int_\Omega |\nabla u|^p \, dx}{\int_\Omega |u|^p \, dx}
\quad \text{and} \quad
\int_\Omega f u \, dx=0.
$$
Comparing these equalities with the definition \eqref{eq:l*} of $\lambda_f^*$, we directly get the first part of the following result.
\begin{prop}\label{prop:q}
	Let $p>1$. If $\lambda < \lambda_f^*$, then \eqref{D} has no solution $u$ such that $E_\lambda(u)=0$. 
	Moreover, the following assertions hold:
	\begin{enumerate}[label={\rm(\roman*)}]
		\item\label{prop:q:1} 
			If $\lambda < \lambda_1$, then, in addition to \ref{MP}, any solution $u$ of \eqref{D} satisfies $E_\lambda(u) < 0$.
		\item\label{prop:q:2}
			If $\lambda_1 < \lambda < \lambda_f^*$, then any solution $u$ of \eqref{D} satisfies $E_\lambda(u) > 0$.
	\end{enumerate}	
\end{prop}

More can be said if we consider the branch of ground state solutions of \eqref{D} with $\lambda \in [\lambda_f^*,\lambda_2)$.
By the \textit{ground state solution} of \eqref{D} we mean a solution $u$ which satisfies
$$
E_\lambda(u) \leq E_\lambda(v)
\quad \text{for any other solution } v \text{ of } \eqref{D}.
$$	
It can be easily seen that \eqref{D} possesses a ground state solution for any $\lambda \in [\lambda_f^*,\lambda_2)$, see, e.g., Lemma \ref{lem:ground_state} below.
In the following theorem, we provide some qualitative properties of the corresponding branch $(\lambda,u) \in [\lambda_f^*,\lambda_2) \times W_0^{1,p}(\Omega)$, which reveals the connection with Theorem \ref{thm:1}.
\begin{thm}\label{thm:2}
	Let $p>1$. The following assertions hold:
	\begin{enumerate}[label={\rm(\roman*)}]
		\item\label{thm:2:2} 
		If $\lambda = \lambda_f^* < \lambda_2$, then any ground state solution $u$ of \eqref{D} is sign-changing and satisfies $E_\lambda(u) = 0$. Moreover, $u$ is a minimizer for $\lambda_f^*$.
		\item\label{thm:2:3} 
		Let $f>0$. If $\lambda_f^* < \lambda < \lambda_2$, then any ground state solution $u$ of \eqref{D} is sign-changing and satisfies $E_\lambda(u) < 0$.
	\end{enumerate}
\end{thm}

We schematically depict the results of Theorem \ref{thm:1} \ref{thm:1:2}, Proposition \ref{prop:q}, and Theorem \ref{thm:2} (taking into account the nonexistence result \cite[Theorem 1]{FGTT} for {\renewcommand\nlab{\lambda_1}\eqref{D}}) on Figure \ref{fig1} below.
\begin{figure}[ht]
	\centering
	\includegraphics[width=0.6\linewidth]{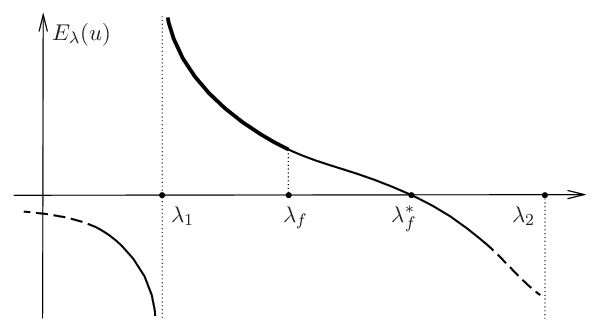}\\
	\caption{Dependence of the energy $E_\lambda(u)$ of the ground state solution $u$ of \eqref{D} on $\lambda$, provided $f>0$ and $\lambda_f^* < \lambda_2$.}
	\label{fig1}
\end{figure}

This paper is organized as follows. In Section \ref{sec:prop}, we prove Proposition \ref{thm:3} together with some other properties of $\lambda_f$ and $\lambda_f^*$. In Section \ref{sec:auxiliary}, we discuss some auxiliary results needed for the proof of Theorems \ref{thm:1} and \ref{thm:2}, and Proposition \ref{prop:q}. Then, in Section \ref{sec:proofs}, we prove three latter statements. Finally, Section \ref{sec:discuss} is devoted to some remarks and discussion.

\section{Some properties of \texorpdfstring{$\lambda_f^*$}{lambda*}}\label{sec:prop}

Let us provide several basic properties of $\lambda_f$ and $\lambda_f^*$. 
In particular, the following lemma contains Proposition \ref{thm:3} and the assertion $\lambda_f^* \in (\lambda_1, \lambda_2]$ of Theorem \ref{thm:1}.

\begin{lemma}\label{lem:l<l<l}
	Let $p>1$. 
	The following assertions hold:
	\begin{enumerate}[label={\rm(\roman*)}]
		\item\label{lem:l<l<l:0} 
		$\lambda^*_f$ possesses a minimizer;
		\item\label{lem:l<l<l:1} 
		$\lambda_1 < \lambda^*_f \leq \lambda_2$;
		\item\label{lem:l<l<l:4}
		if there exists a second eigenfunction $\varphi_2$ such that $\int_\Omega f \varphi_2 \, dx \neq 0$, then $\lambda_f^* < \lambda_2$; 
		\item\label{lem:l<l<l:45}
		if $p=2$ and $N=1$, then there exists $f$ such that $\int_\Omega f \varphi_2 \, dx = 0$ and $\lambda_f^* < \lambda_2$;
		\item\label{lem:l<l<l:2} 
		if $p=2$, then $\lambda_f = \lambda^*_f = \lambda_2$ for $f = \varphi_1$ (modulo scaling);
		\item\label{lem:l<l<l:3} 
		there exists a sequence $\{f_n\} \subset L^\infty(\Omega)$ such that $\lambda^*_{f_n} \to \lambda_1$ as $n \to +\infty$.
	\end{enumerate}
\end{lemma}
\begin{proof}
	\ref{lem:l<l<l:0} The existence of a minimizer for $\lambda_f^*$ can be easily proved by standard arguments. By definition, any minimizer $u$ satisfies $u\not\equiv 0$ and $\int_\Omega f u \,dx = 0$.
	
	\ref{lem:l<l<l:1} Since $f \geq 0$, $f \not\equiv 0$, and the first eigenfunction $\varphi_1$ is strictly positive,  \cite{vaz}, the inequality $\lambda_1 < \lambda^*_f$ directly follows from assertion \ref{lem:l<l<l:0}.
	Let us show that $\lambda^*_f \leq \lambda_2$.
	Fix any second eigenfunction $\varphi_2 = \varphi_2^+ + \varphi_2^-$. 
	Assume first that $\int_\Omega f \varphi_2^\pm \,dx \neq 0$. Then we can find $\alpha > 0$ such that
	$$
	\int_\Omega f (\alpha \varphi_2^+ + \varphi_2^-) \, dx = 
	\alpha \int_\Omega f \varphi_2^+ \, dx + 
	\int_\Omega f \varphi_2^- \, dx =
	0.
	$$
	Since $\varphi_2^\pm \in W_0^{1,p}(\Omega)$, cf.\ \cite[Lemma 7.6]{giltrud}, $\alpha \varphi_2^+ + \varphi_2^-$ is admissible for the minimization problem \eqref{eq:l*} defining $\lambda_f^*$. 
	Therefore, from
	\begin{equation}
	\label{eq:phi2}
	\int_\Omega |\nabla \varphi_2^\pm|^p \, dx = \lambda_2 \int_\Omega |\varphi_2^\pm|^p \, dx
	\end{equation}
	we deduce that
	\begin{equation}\label{eq:lf<l2}
	\lambda^*_f 
	\leq 
	\frac{\alpha^p \int_\Omega |\nabla \varphi_2^+|^p \, dx + \int_\Omega |\nabla \varphi_2^-|^p \, dx}{\alpha^p \int_\Omega |\varphi_2^+|^p \, dx + \int_\Omega |\varphi_2^-|^p \, dx} = \lambda_2.
	\end{equation}
	Assume now that $\int_\Omega f \varphi_2^+ \,dx = 0$. In this case, $\varphi_2^+$ is admissible for \eqref{eq:l*}, and hence \eqref{eq:phi2} again leads to the desired inequality $\lambda_f^* \leq \lambda_2$. The remaining case $\int_\Omega f \varphi_2^- \,dx = 0$ is similar.
	
	\ref{lem:l<l<l:4}
	Suppose, by contradiction, that there exists a second eigenfunction $\varphi_2$ such that $\int_\Omega f \varphi_2 \, dx \neq 0$ and $\lambda_f^* = \lambda_2$. 
	Assume first that $\int_\Omega f \varphi_2^\pm \, dx \neq 0$. 
	Arguing as in assertion \ref{lem:l<l<l:1}, we can find $\alpha>0$, $\alpha \neq 1$, such that $\int_\Omega f (\alpha \varphi_2^+ + \varphi_2^-) \,dx = 0$, and hence, recalling that $\lambda_f^* = \lambda_2$, we get from \eqref{eq:lf<l2} that $w = \alpha \varphi_2^+ + \varphi_2^-$ is a minimizer for $\lambda_f^*$. Thus, by the Lagrange multiplier rule, we obtain $\mu_1, \mu_2 \in \mathbb{R}$ such that $|\mu_1|+|\mu_2|>0$ and 
	\begin{equation}\label{eq:sol1}
	\mu_1 \left(\int_\Omega |\nabla w|^{p-2} (\nabla w, \nabla \xi) \, dx - \lambda^*_f \int_\Omega |w|^{p-2} w \xi \, dx \right)
	+
	\mu_2 \int_\Omega f \xi \, dx = 0,
	\quad
	\forall \xi \in W_0^{1,p}(\Omega).
	\end{equation}
	If we suppose that $\mu_1 = 0$, then $f \equiv 0$, which is impossible. 
	Therefore, $\mu_1 \neq 0$. 
	Moreover, $\mu_2 \neq 0$ as well. Indeed, if we suppose that $\mu_2=0$, then \eqref{eq:sol1} states that $\alpha \varphi_2^+ + \varphi_2^-$ is a second eigenfunction. However, since $\alpha \neq 1$, $C^{1,\gamma}(\overline{\Omega})$-regularity of both $\varphi_2 = \varphi_2^+ + \varphi_2^-$ and $\alpha \varphi_2^+ + \varphi_2^-$ contradicts the Hopf maximum principle at the boundary between nodal domains, see \cite[Lemma 2.4]{DR}.
	Thus, $\mu_1 \mu_2 \neq 0$	 implies that 
	$$
	v = \left|\frac{\mu_1}{\mu_2}\right|^\frac{1}{p-1} \text{sign}\left(\frac{\mu_1}{\mu_2}\right) (\alpha \varphi_2^+ + \varphi_2^-)
	$$
	is a solution of \eqref{D}. But this is again impossible in view of \cite[Lemma 2.4]{DR}, since $\alpha \neq 1$ and any solution of \eqref{D} belongs to $C^{1,\gamma}(\overline{\Omega})$. 
	Suppose now that $\int_\Omega f \varphi_2^+ \, dx = 0$. Thus, $w = \varphi_2^+$ is a minimizer for $\lambda_f^* = \lambda_2$ and hence it again satisfies \eqref{eq:sol1} with some Lagrange multipliers $\mu_1 \neq 0$ and $\mu_2 \in \mathbb{R}$. That is, either $w$ is a second eigenfunction (provided $\mu_2=0$) or a proper scaling of $w$ is a solution of \eqref{D} (provided $\mu_2 \neq 0$). In both cases, we get a contradiction, since the Hopf maximum principle implies $w \not\in C^{1,\gamma}(\overline{\Omega})$. 
	The case $\int_\Omega f \varphi_2^- \, dx = 0$ is similar.
	
	\ref{lem:l<l<l:45} 
	Let us consider the following simple example:
	\begin{equation}\label{eq:D-with-f}
	\left\{
	\begin{aligned}
	-u'' &= \lambda u + (1-\sin(x)) &&\text{in } (0, \pi),\\
	u(0)&=u(\pi)=0.
	\end{aligned}
	\right.
	\end{equation}
	Note that $f(x) = (1-\sin(x)) > 0$ a.e.\ in $(0,\pi)$. 
	Moreover, since $\varphi_2=\sin(2x)$, we have $\int_0^\pi f \varphi_2 \,dx =0$. 
	One can easily find an explicit solution $u_\lambda$ of \eqref{eq:D-with-f} and obtain that $\int_0^\pi f u_{\lambda_0} \, dx = 0$ for some $\lambda_0 \in (3,3.5) \subset (\lambda_1,\lambda_2)$. 
	Thus, $\lambda_f^* \leq \lambda_0 < \lambda_2$.	
	
	\ref{lem:l<l<l:2} 
	If $p=2$ and $f = \varphi_1$, then the definition \eqref{eq:l*} of $\lambda_f^*$ coincides with the definition of the second eigenvalue $\lambda_2$ by the Courant-Fisher variational principle, that is, $\lambda_f^* = \lambda_2$.
	The equality $\lambda_f = \lambda^*_f$ can be also observed easily by the Fredholm alternative. Indeed, consider the problem \eqref{D} with $p=2$ and $f=\varphi_1$, i.e., 
	\begin{equation}\label{eq:D-with-first-eigenfunction}
	\left\{
	\begin{aligned}
	-\Delta u &= \lambda u + \varphi_1(x) &&\text{in } \Omega,\\
	u&=0 &&\text{on } \partial \Omega.
	\end{aligned}
	\right.
	\end{equation}
	If we try to find a solution of \eqref{eq:D-with-first-eigenfunction} in the form $u = C \varphi_1$, where $C \in \mathbb{R}$, then we obtain that 
	$$
	u(x) = -\frac{1}{\lambda-\lambda_1} \, \varphi_1(x),
	\quad 
	x \in \Omega.
	$$
	This shows that $u$ is strictly negative whenever $\lambda > \lambda_1$. By the Fredholm alternative, $u$ is the unique solution of \eqref{eq:D-with-first-eigenfunction} for $\lambda \in (\lambda_1,\lambda_2)$, and if $\lambda = \lambda_2$, then $u + \alpha \varphi_2$ is a solution of \eqref{eq:D-with-first-eigenfunction} for any $\alpha \in \mathbb{R}$ and $\varphi_2$. Since each $\varphi_2$ is sign-changing, we can take $|\alpha_0|$ large enough in order to conclude that $u + \alpha_0 \varphi_2$ also changes sign. That is, \ref{AMP} is valid with $\lambda_f = \lambda_f^* = \lambda_2$. 
	
	\ref{lem:l<l<l:3} 
	Since $\varphi_1 \in W_0^{1,p}(\Omega)$, there exists a sequence $\{v_n\} \subset C_0^\infty(\Omega)$ such that $v_n \to \varphi_1$ strongly in $W_0^{1,p}(\Omega)$, that is,
	\begin{equation}\label{eq:convergence1}
	\int_\Omega |\nabla v_n|^p \, dx \to 
	\int_\Omega |\nabla \varphi_1|^p \, dx
	\quad \text{and} \quad
	\int_\Omega |v_n|^p \, dx \to 
	\int_\Omega |\varphi_1|^p \, dx
	\end{equation}
	as $n \to +\infty$.
	Moreover, we can assume that each $v_n \geq 0$.
	
	Since $\text{supp}\, v_n$ is compactly contained inside of $\Omega$, we can find a sufficiently small ball $B_n \subset \Omega \setminus \text{supp}\, v_n$. 
	Let $\xi_n \in C_0^\infty(\Omega)$ be such that $\xi_n \leq 0$, $\text{supp}\, \xi_n \subset B_n$, and $\int_\Omega |\nabla \xi_n|^p \, dx = 1$ for each $n \in \mathbb{N}$.
	Taking any sequence $\{\beta_n\} \subset (0,+\infty)$ such that $\beta_n \to 0$ as $n \to +\infty$, we deduce from \eqref{eq:convergence1} that 
	\begin{equation}\label{eq:limit-to-phi1}
	\frac{\int_\Omega |\nabla (v_n + \beta_n \xi_n)|^p \, dx}{\int_\Omega |v_n + \beta_n \xi_n|^p \, dx} 
	=
	\frac{\int_\Omega |\nabla v_n|^p \, dx + \beta_n^p}{\int_\Omega |v_n|^p \, dx + \beta_n^p \int_\Omega |\xi_n|^p \, dx}
	\to 
	\lambda_1
	\quad
	\text{as }
	n \to +\infty.
	\end{equation}
	
	Now, we take for each $n \in \mathbb{N}$ any function $f_n \in L^\infty(\Omega)$ such that $f_n = 1$ on $\text{supp}\, v_n$, $f_n = a_n > 0$ on $B_n$, and $f_n \geq 0$ on $\Omega \setminus \left(\text{supp}\, v_n \cup B_n\right)$. Adjusting a constant $a_n>0$ in such a way that
	$$
	\int_\Omega f_n (v_n + \beta_n \xi_n) \, dx = 
	\int_{\text{supp}\, v_n} v_n \, dx + 
	a_n \,\beta_n \int_{B_n} \xi_n \, dx = 0,
	$$
	we make $v_n + \beta_n \xi_n$ admissible for \eqref{eq:l*}. Thus, \eqref{eq:limit-to-phi1} implies that $\lambda_{f_n}^* \to \lambda_1$ as $n \to +\infty$. 
\end{proof}

\smallskip
The following important property of $\lambda^*_f$ can be obtained directly from \eqref{eq:l*}.
\begin{lemma}\label{lem:H>0}
	Let $p>1$, and let $u \in W_0^{1,p}(\Omega)\setminus \{0\}$ be such that $\int_\Omega f u \, dx = 0$. 
	If $\lambda < \lambda^*_f$, then $H_\lambda(u) > 0$.
	If $\lambda = \lambda^*_f$, then $H_\lambda(u) \geq 0$. 
\end{lemma}

\section{Auxiliary results}\label{sec:auxiliary}
Let us state several auxiliary results which will be used in Section \ref{sec:proofs} below. 
Note that any solution of \eqref{D} belongs to the Nehari manifold
$$
\mathcal{N}_\lambda := \left\{
w \in W_0^{1,p}(\Omega) \setminus \{0\}:~
H_\lambda(w) = \int_\Omega f w \, dx
\right\}.
$$
Since we assume that the function $f$ is fixed, we do not indicate the dependence of the Nehari manifold on $f$ by an additional index in the notation $\mathcal{N}_\lambda$. 
Let us collect some basic properties of $\mathcal{N}_\lambda$.
\begin{lemma}\label{lem:neh}
	Let $p>1$. 
	The following assertions hold:
	\begin{enumerate}[label={\rm(\roman*)}]
		\item\label{lem:neh:3}
		If $u \in W_0^{1,p}(\Omega)$ satisfies 
		\begin{equation}\label{eq:lem:neh:3}
		H_\lambda(u) 
		\cdot
		\int_\Omega f u \, dx > 0,
		\end{equation}
		then there exists a unique $t_u > 0$ such that $t_u u \in \mathcal{N}_\lambda$ and
		\begin{equation}\label{eq:tu}
		t_u = \frac{\left|\int_\Omega f u \, dx\right|^\frac{1}{p-1}}{\left|H_\lambda(u)\right|^\frac{1}{p-1}},
		\quad
		E_\lambda(t_u u) = 
		\left(\frac{1}{p}-1\right) 
		\frac{\left|\int_\Omega f u \, dx\right|^\frac{p}{p-1}}{\left|H_\lambda(u)\right|^\frac{1}{p-1}}\, \mathrm{sign}\left(H_\lambda(u)\right).
		\end{equation}
		Moreover, if $H_\lambda(u) < 0$, then $t_u$ is a point of global maximum of $E_\lambda(tu)$ with respect to $t>0$. 
		If $H_\lambda(u) > 0$, then $t_u$ is a point of global minimum of $E_\lambda(tu)$ with respect to $t>0$. 
		\item\label{lem:neh:2} 
		$\mathcal{N}_\lambda \neq \emptyset$ for any $\lambda \in \mathbb{R}$.
		\item\label{lem:neh:1}
		If $u \in \mathcal{N}_\lambda$, then $\left.\frac{\partial}{\partial t} E_\lambda(t u) \right|_{t=1} = 0$ and 
		\begin{equation}\label{eq:EonNeh}
		E_\lambda(u) 
		= \left(\frac{1}{p}-1\right) 
		H_\lambda(u)
		=
		\left(\frac{1}{p}-1\right) \int_\Omega f u \, dx.
		\end{equation}
	\end{enumerate}
\end{lemma}
\begin{proof}
	To prove assertion \ref{lem:neh:3}, for fixed $u \in W_0^{1,p}(\Omega)$ consider the fibering functional	
	$$
	E_\lambda(tu) = \frac{t^p}{p} 
	H_\lambda(u)
	-
	t \int_\Omega f u \, dx, 
	\quad
	t > 0.
	$$
	If $u$ satisfies \eqref{eq:lem:neh:3}, then we can find $t_u > 0$ such that $\left.\frac{\partial}{\partial t}E_\lambda(tu)\right|_{t=t_u}=0$, which yields $t_u u \in \mathcal{N}_\lambda$. Other properties of $t_u$ trivially follow from the definition of $E_\lambda(tu)$.
	
	Let us prove assertion \ref{lem:neh:2}. Since $f \not\equiv 0$, we can find a function $u \in C_0^\infty(\Omega)$ with the support of arbitrarily small measure such that $\int_\Omega f u \,dx \neq 0$, or, without loss of generality, $\int_\Omega f u \,dx > 0$. 
	By the Poincar\'e inequality \cite[(7.44)]{giltrud} we have
	$$
	\int_\Omega |\nabla u|^p \,dx \geq \left(\frac{\omega_N}{|\text{supp}\, u|}\right)^\frac{p}{N} \int_\Omega |u|^p \,dx,
	$$
	where $\omega_N$ is the volume of a unit ball in $\mathbb{R}^N$.
	Therefore, for any fixed $\lambda \in \mathbb{R}$ we can ask for the support of $u$ to be of sufficiently small measure in order to get $H_\lambda(u)>0$. Then, assertion \ref{lem:neh:2} follows by applying  assertion \ref{lem:neh:3} to $u$.
	
	Assertion \ref{lem:neh:1} is trivial. 	
\end{proof}

The following lemma is a consequence of \cite[Theorem 2.1]{AH} and \cite[Theorem 5]{vaz}.
\begin{lemma}\label{lem:nonnegative}
	Let $p>1$ and $\lambda > \lambda_1$. Then \eqref{D} does not possess nonnegative solutions. That is, any solution is either nonpositive or sign-changing.
\end{lemma}

\begin{lemma}\label{lem:sign-changing}
	Let $p>1$. If $u$ is a solution of \eqref{D} and $\int_\Omega f u \, dx = 0$, then $u$ is sign-changing.
\end{lemma}
\begin{proof}
	Let $u$ be a solution of \eqref{D} for some $\lambda \in \mathbb{R}$ such that $\int_\Omega f u \, dx = 0$. Since $u \in \mathcal{N}_\lambda$, we have $H_\lambda(u)=0$, and hence Lemma \ref{lem:H>0} implies that $\lambda \geq \lambda_f^* > \lambda_1$. Therefore, $u$ is either nonpositive or sign-changing, as it follows from Lemma \ref{lem:nonnegative}. 
	Suppose, by contradiction, that $u \leq 0$.
	Since $\int_\Omega f u \, dx = 0$ and $f \geq 0$, we have $f \equiv 0$ on $A:= \{x \in \Omega:\, u(x) < 0\}$.
	Due to $C^{1,\gamma}(\overline{\Omega})$-regularity of $u$, the set $A$ is open. 
	Therefore, $\lambda$ and $u$ are the first eigenvalue and the first eigenfunction of \eqref{E} on the domain $A$, respectively. 
	Since $f$ is nontrivial, we have $f \not\equiv 0$ on $\Omega \setminus A$, which implies $\partial A \cap \Omega \neq \emptyset$. 
	Let a ball $B \subset \overline{A}$ and a point $x_0 \in \Omega$ be such that $x_0 \in \partial B \cap \partial A \cap \Omega$. Then $u(x_0)=0$, and by the Hopf maximum principle $\frac{\partial u(x_0)}{\partial \nu} > 0$ (see, e.g., \cite[Theorem 5]{vaz}), where $\nu$ is an outward unit normal to $\partial B$. However, since $u \leq 0$ in the whole of $\Omega$, we obtain a contradiction with $C^{1,\gamma}(\overline{\Omega})$-regularity of $u$. 
	Therefore, $u$ changes sign.
\end{proof}

Finally, the following lemma is a consequence of the definition of the second eigenvalue, see, e.g., \cite[Remark 4]{BP}.
\begin{lemma}\label{lem:second_eigenvalue}
	Let $p>1$. If there exists $u \in W_0^{1,p}(\Omega)$ such that $u^\pm \not\equiv 0$ and $H_\lambda(u^\pm) \leq 0$, then $\lambda \geq \lambda_2$.
\end{lemma}

\section{Proofs of the main results}\label{sec:proofs}

First, we prove a weak form of Theorem \ref{thm:1} \ref{thm:1:2} by showing the nonstrict inequality $\lambda_f \leq \lambda_f^*$. Then, we prove Theorem \ref{thm:1} \ref{thm:1:1}, Proposition \ref{prop:q}, and Theorem \ref{thm:2}. For technical convenience, we postpone the completion of the proof of Theorem \ref{thm:1} \ref{thm:1:2} to the end of this section.

We start with the following lemma.  
Recall that $\lambda_f^*$ defined by \eqref{eq:l*} possesses a minimizer, see Lemma \ref{lem:l<l<l} \ref{lem:l<l<l:0}.

\begin{lemma}\label{prop:existence_zero}
	Let $p>1$. If $\lambda_f^* < \lambda_2$, then any minimizer $u$ for $\lambda_f^*$ is a sign-changing solution of {\renewcommand\nlab{\lambda_f^*}\eqref{D}} and satisfies $E_{\lambda_f^*}(u) = 0$.
\end{lemma}
\begin{proof}
	Let $u$ be a minimizer for $\lambda_f^*$. In particular, $u \not\equiv 0$ and $\int_\Omega f u \, dx = 0$. 
	By the Lagrange multiplier rule, there exist $\nu_1, \nu_2 \in \mathbb{R}$ such that $|\nu_1|+|\nu_2|>0$ and 
	\begin{equation}\label{eq:sol}
	\nu_1 \left(\int_\Omega |\nabla u|^{p-2} (\nabla u, \nabla \xi) \, dx - \lambda^*_f \int_\Omega |u|^{p-2} u \xi \, dx \right)
	+
	\nu_2 \int_\Omega f \xi \, dx = 0,
	\quad
	\forall \xi \in W_0^{1,p}(\Omega).
	\end{equation}
	Note that $\nu_1 \neq 0$, since otherwise \eqref{eq:sol} yields $f \equiv 0$, which contradicts our assumptions on $f$. 
	Suppose that $\nu_2 = 0$. In this case, \eqref{eq:sol} implies that $\lambda^*_f$ is an eigenvalue of \eqref{E} and $u$ is an eigenfunction associated with $\lambda_f^*$. 
	Recalling that $\lambda_1 < \lambda^*_f \leq \lambda_2$ by Lemma \ref{lem:l<l<l} \ref{lem:l<l<l:1}, we conclude that $\lambda^*_f = \lambda_2$. However, it contradicts our assumption $\lambda^*_f < \lambda_2$. Therefore $\nu_2 \neq 0$, and hence we deduce from \eqref{eq:sol} that the function 
	$$
	v = \left|\frac{\nu_1}{\nu_2}\right|^\frac{1}{p-1} \text{sign}\left(\frac{\nu_1}{\nu_2}\right) u
	$$
	is a solution of {\renewcommand\nlab{\lambda_f^*}\eqref{D}}. 
	Since $\int_\Omega f v \,dx = 0$ and $v \in \mathcal{N}_\lambda$, we get from Lemma \ref{lem:neh} \ref{lem:neh:1} that $E_{\lambda_f^*}(v)=0$, and Lemma \ref{lem:sign-changing} implies that $v$ is sign-changing.
\end{proof}

The following direct corollary of Lemma \ref{prop:existence_zero} gives us a weak form of Theorem \ref{thm:1} \ref{thm:1:2}.
\begin{cor}\label{cor:l<l}
	Let $p>1$. If $\lambda_f^* < \lambda_2$, then $\lambda_f \leq \lambda_f^*$.
\end{cor}

Let us now prove Theorem \ref{thm:1} \ref{thm:1:1}.
Hereinafter, $\|u\| := \left(\int_\Omega |u|^p \,dx \right)^{1/p}$ stands for the norm of $u$ in $L^p(\Omega)$.
\begin{lemma}\label{lem:N=1}
	Let $p>1$. If $p=2$ or $N=1$, then $\lambda_f \leq \lambda_f^*$.
\end{lemma}
\begin{proof}
	Assume first that $p=2$. If $\lambda_f^* < \lambda_2$, then the desired conclusion is given by Corollary \ref{cor:l<l}. Suppose that $\lambda_f^* = \lambda_2$. 
	Then Lemma \ref{lem:l<l<l} \ref{lem:l<l<l:4} implies that $\int_\Omega f \varphi_2 \, dx = 0$ for any second eigenfunction $\varphi_2$. 
	Therefore, {\renewcommand\nlab{\lambda_f^*}\eqref{D}} has a solution $w$ by the Fredholm alternative. 
	If $w$ is sign-changing, then $\lambda_f \leq \lambda_f^*$, and hence we are done. Let $w$ has a constant sign. Note that for any $\alpha \in \mathbb{R}$ and any $\varphi_2$, $w + \alpha \varphi_2$ is also a solution of {\renewcommand\nlab{\lambda_f^*}\eqref{D}}. Since any $\varphi_2$ is sign-changing, we can take $|\alpha|$ large enough in order to deduce that $w + \alpha \varphi_2$ also changes sign, and the proof for the case $p=2$ is finished.
	
	Assume now that $p>1$ and $N=1$. 
	Recall that if $\lambda_f^* < \lambda_2$, then the assertion is given by Corollary \ref{cor:l<l}. Suppose that $\lambda_f^* = \lambda_2$.
	By the definition of $\lambda_f$, the inequality $\lambda_f \leq \lambda_f^*$ will be established if we prove the existence of a sign-changing solution of \eqref{D} in an arbitrarily small neighbourhood of $\lambda_2$. 	
	Assume, without loss of generality, that $\Omega = (0,1)$. 
	We easily see that $v$ is a solution of \eqref{D} if and only if $u = \|v'\|^{-2} v$ is a solution of the problem
	\begin{equation}\label{D1d}
	\tag*{(\theequation)$_{\nlabb}$}
	\left\{
	\begin{aligned}
	-(|u'|^{p-2}u')' &= \lambda |u|^{p-2} u + \|u'\|^{2(p-1)} f(x) &&\text{in } (0,1),\\
	u(0)&=u(1)=0.
	\end{aligned}
	\right.
	\end{equation}
	To obtain the claim, we will sketchily show that $(\lambda_2, 0) \in \mathbb{R} \times W_0^{1,p}(0,1)$ is a bifurcation point for \ref{D1d}. 
	For a given $h \in W^{-1,p'}(0,1)$, let $R(h) \in W_0^{1,p}(0,1)$ be a unique solution of the problem
	\begin{equation*}\label{D1d2}
	\left\{
	\begin{aligned}
	-(|u'|^{p-2}u')' &= h &&\text{in } (0,1),\\
	u(0)=u(1)&=0.
	\end{aligned}
	\right.
	\end{equation*}
	Then $R$ seen as a map $L^{r}(0,1) \mapsto W_0^{1,p}(0,1)$ is completely continuous for any $r>1$,  \cite[p.~229]{delpinoman}.
	Note that $u$ is a solution of \ref{D1d} if and only if $u$ satisfies the operator equation 
	$$
	u=R\left(\lambda |u|^{p-2}u + F(u)\right),
	$$
	where the composition operator $F: W_0^{1,p}(0,1) \mapsto L^\infty(0,1)$ is defined by $F(u)(x) = \|u'\|^{2(p-1)} f(x)$.
	It is not hard to see that $H_\lambda(u) := R\left(\lambda |u|^{p-2}u + F(u)\right)$ defines a completely continuous map $W_0^{1,p}(0,1) \mapsto W_0^{1,p}(0,1)$. That is, $I-H_\lambda: W_0^{1,p}(0,1) \mapsto W_0^{1,p}(0,1)$ is a continuous compact perturbation of the identity $I$ in $W_0^{1,p}(0,1)$.
	
	Arguing as in \cite[Proposition 2.2]{delpinoman} (see also \cite[Theorem 4.1]{delman}), we see that for any $\delta>0$ the following index formula is satisfied:
	$$
	\text{deg}_{W_0^{1,p}(0,1)}\left(I - R(\lambda |\cdot|^{p-2} \cdot), B(0,\delta), 0\right) = 
	\left\{
	\begin{aligned}
	-&1 &\text{for } \lambda \in (\lambda_1, \lambda_2),\\
	&1 &\text{for } \lambda \in (\lambda_2,\lambda_3),
	\end{aligned}
	\right.
	$$
	where ``$\text{deg}$'' denotes the Leray-Schauder degree, $B(0,\delta)$ is a ball in $W_0^{1,p}(0,1)$ with radius $\delta$ centred at $0$, $\lambda_3$ is the third eigenvalue of the one-dimensional $p$-Laplacian, and $\lambda_2 < \lambda_3$ by, e.g., \cite[Section 3]{delman}.
	
	Using the above-mentioned facts, we can argue in much the same way as in \cite[Theorem 1.1]{delpinoman} to obtain that $(\lambda_2,0)$ is a bifurcation point for \ref{D1d}. 
	Moreover, if $\mu_k \to \lambda_2$ as $k \to +\infty$ and $u_k$ is a solution of 
	{\renewcommand\nlabb{\mu_k}\ref{D1d}}
	belonging to the bifurcation branch emanating from $(\lambda_2,0)$, then $u_k/\|u_k'\| \to \varphi_2$ in $W_0^{1,p}(0,1)$ up to a subsequence, see \cite[p.~231]{delpinoman}. Thus, since $\varphi_2$ is sign-changing, $u_k$ is also sign-changing for sufficiently large $k$. Therefore, we conclude that $v_k = \|u_k'\|^{-2} u_k$ is a sign-changing solution of {\renewcommand\nlab{\mu_k}\eqref{D}} for such $k$, and the proof is complete.	
\end{proof}

Now we are going to prove Proposition \ref{prop:q}. Assertion \ref{prop:q:1} of Proposition \ref{prop:q} directly follows from the combination of Lemma \ref{lem:neh} \ref{lem:neh:1} with the fact that $H_\lambda(v) > 0$ for any $v \in W_0^{1,p}(\Omega) \setminus \{0\}$ provided $\lambda < \lambda_1$.
Assertion \ref{prop:q:2} of Proposition \ref{prop:q} follows from the combination of Lemma \ref{lem:neh} \ref{lem:neh:1} with the first part of the following result.
\begin{lemma}\label{lem:H<0}
	Let $p>1$. If $\lambda \in (\lambda_1, \lambda^*_f)$, then any solution $u$ of \eqref{D} satisfies $H_\lambda(u) < 0$.
	If $\lambda^*_f < \lambda_2$, then any solution $u$ of {\renewcommand\nlab{\lambda_f^*}\eqref{D}} satisfies $H_{\lambda_f^*}(u) \leq 0$.
\end{lemma}
\begin{proof}
	Suppose first, by contradiction, that there exists $\lambda \in (\lambda_1, \lambda^*_f)$ and a solution $u$ of \eqref{D} such that $H_\lambda(u) \geq 0$.
	Actually, the equality here is impossible, since otherwise $u \in \mathcal{N}_\lambda$ gives
	$\int_\Omega f u \, dx = 0$, which contradicts Lemma \ref{lem:H>0}.
	Therefore, $H_\lambda(u) > 0$ and consequently $\int_\Omega f u \, dx > 0$ due to $u \in \mathcal{N}_\lambda$. 
	Moreover, the same assumptions occur if we argue by contradiction in the case $\lambda = \lambda^*_f < \lambda_2$. 
	
	The inequality $\int_\Omega f u \, dx > 0$ implies $\int_\Omega f u^+ \, dx > 0$, and hence $u$ cannot be nonpositive. Moreover, Lemma \ref{lem:nonnegative} implies that $u$ cannot be nonnegative, as well.
	Thus, $u$ is a sign-changing function. Assume first that $\int_\Omega f u^- \, dx < 0$. Then
	we can find $\alpha \in (0,1)$ such that
	\begin{equation}\label{eq:lem:H<0:1}
	\int_\Omega f (\alpha u^+ + u^-) \, dx = \alpha \int_\Omega f u^+ \, dx + \int_\Omega f u^- \, dx = 0.
	\end{equation}
	On the other hand, since $u$ is a solution of \eqref{D}, we get
	\begin{equation}\label{eq:lem:H<0:2}
	H_\lambda(u^\pm) = \int_\Omega f u^\pm \, dx.
	\end{equation}
	Combining \eqref{eq:lem:H<0:2} with \eqref{eq:lem:H<0:1}, we conclude that
	\begin{align*}
	H_\lambda(\alpha u^+ + u^-) 
	&= \alpha^p H_\lambda(u^+) + H_\lambda(u^-) = \\
	&= \alpha^p \int_\Omega f u^+ \, dx + \int_\Omega f u^- \, dx = 
	(\alpha^p - \alpha) \int_\Omega f u^+ \, dx < 0,
	\end{align*}
	since $\alpha \in (0,1)$. However, it contradicts Lemma \ref{lem:H>0} applied to the function $\alpha u^+ + u^-$.
	Assume now that $\int_\Omega f u^- \, dx = 0$. In this case we again get a contradiction with Lemma \ref{lem:H>0} in view of \eqref{eq:lem:H<0:2}, provided $\lambda \in (\lambda_1,\lambda_f^*)$. 
	If $\lambda=\lambda_f^* < \lambda_2$, then we see that $u^-$ is a minimizer for $\lambda_f^*$ and hence Lemma \ref{prop:existence_zero} gives a contradiction.
\end{proof}

\begin{remark}
	Note that for $\lambda \in (\lambda_1, \lambda_f)$ we have $u<0$ by \ref{AMP}, which yields $\int_\Omega f u \, dx < 0$, and hence the result of Lemma \ref{lem:H<0} simply follows from Lemma \ref{lem:neh} \ref{lem:neh:1}.
	That is, the result of Lemma \ref{lem:H<0} is nontrivial only for $\lambda \in [\lambda_f, \lambda_f^*]$. 
\end{remark}

Let us now prove Theorem \ref{thm:2}. 
First, we show that \eqref{D} has a ground state solution provided $\lambda$ is not an eigenvalue of \eqref{E}.
With a slight abuse of notation, we will write $\|\nabla u\| := \left(\int_\Omega |\nabla u|^p \,dx \right)^{1/p}$ for the norm of $u$ in $W_0^{1,p}(\Omega)$.

\begin{lemma}\label{lem:ground_state}
	Let $p>1$. If $\lambda$ is not an eigenvalue of \eqref{E}, then \eqref{D} possesses a ground state solution.
\end{lemma}
\begin{proof}
	By the results of \cite[Theorem 3.1, p.\ 60]{FNSS} or \cite[Theorem 3]{pokhozhaev1967}, \eqref{D} possesses a solution for any $\lambda$ which is not an eigenvalue of \eqref{E}.
	If this solution is unique (as it is for $p=2$) or if there are finitely many solutions, then we are done. 
	Let us assume that for some admissible $\lambda$ there is an infinite sequence of solutions $\{v_n\} \subset W_0^{1,p}(\Omega)$ of \eqref{D} such that
	$$
	E_\lambda(v_n) 
	\to 
	\inf\left\{E_\lambda(w):~ w \text{ is a solution of } \eqref{D}\right\}.
	$$
	If $\{v_n\}$ is bounded in $W_0^{1,p}(\Omega)$, then, up to a subsequence, $\{v_n\}$ converges strongly in $W_0^{1,p}(\Omega)$ to a solution of \eqref{D} since $E_\lambda$ satisfies the Palais--Smale condition provided $\lambda$ is not an eigenvalue of \eqref{E}.
	On the other hand, if we suppose that, up to a subsequence, $\|\nabla v_n\| \to +\infty$ as $n \to +\infty$, then the normalized sequence consisted of $\tilde{v}_n:=\frac{v_n}{\|\nabla v_n\|}$ converges (again up to a subsequence) weakly in $W_0^{1,p}(\Omega)$ and strongly in $L^p(\Omega)$ to some $\tilde{v} \in W_0^{1,p}(\Omega)$.
	Moreover, each $\tilde{v}_n$ satisfies the equation
	$$
	-\Delta_p \tilde{v}_n = \lambda |\tilde{v}_n|^{p-2}\tilde{v}_n + \frac{f(x)}{\|\nabla v_n\|^{p-1}}
	\quad
	\text{in }
	\Omega,
	$$
	in the weak sense.
	By the Nehari constraint, $v_n \in \mathcal{N}_\lambda$ for any $n \in \mathbb{N}$, and
	$$
	H_\lambda(\tilde{v}_n) = 1 - \lambda \|\tilde{v}_n\|^p = \frac{\int_\Omega f \tilde{v}_n \, dx}{\|\nabla v_n\|^{p-1}} \to 0 
	\quad
	\text{as }
	n \to +\infty.
	$$
	Consequently, $\|\tilde{v}_n\|$ does not converge to $0$, which implies that $\tilde{v} \not\equiv 0$. 
	Thus, $\{\tilde{v}_n\}$ converges weakly to an eigenfunction of $-\Delta_p$, and hence $\lambda$ is the associated eigenvalue, which is impossible. 
\end{proof}

Assertion \ref{thm:2:2} of Theorem \ref{thm:2} is based on Lemmas \ref{prop:existence_zero} and \ref{lem:H<0} and follows from the following result.
\begin{lemma}\label{lem:ground_state_0}
	Let $p>1$ and $\lambda_f^* < \lambda_2$. Then any ground state solution $u$ of {\renewcommand\nlab{\lambda_f^*}\eqref{D}} is sign-changing and satisfies $E_{\lambda_f^*}(u) = 0$.	
\end{lemma}
\begin{proof}
	Since $\lambda_1 < \lambda_f^* < \lambda_2$, there exists a ground state solution $v$ of {\renewcommand\nlab{\lambda_f^*}\eqref{D}} by Lemma \ref{lem:ground_state}. The second part of Lemma \ref{lem:H<0} in combination with Lemma \ref{lem:neh} \ref{lem:neh:1} implies that $E_{\lambda_f^*}(v) \geq 0$. 
	Moreover, we know from Lemma \ref{prop:existence_zero} that {\renewcommand\nlab{\lambda_f^*}\eqref{D}} possesses a sign-changing solution $u$ with $E_{\lambda_f^*}(u)=0$. 
	That is, $u$ is a ground state solution. 
	Therefore, any other ground state solution $v$ of {\renewcommand\nlab{\lambda_f^*}\eqref{D}} also satisfies $E_{\lambda_f^*}(v)=0$, and hence Lemma \ref{lem:sign-changing} implies that $v$ is sign-changing.
\end{proof}

Let us complete the proof of Theorem \ref{thm:2} by obtaining assertion \ref{thm:2:3}. 
This assertion will follow if we show that for $\lambda \in (\lambda_f^*,\lambda_2)$ there exists a solution $v$ of \eqref{D} such that $E_\lambda(v) < 0$. In this case, Lemma \ref{lem:ground_state} implies that any ground state solution $u$ of \eqref{D} also satisfies $E_\lambda(u) < 0$. 
Consequently, $\int_\Omega f u \,d x > 0$ by Lemma \ref{lem:neh} \ref{lem:neh:1}, and $u$ cannot be nonnegative by Lemma \ref{lem:nonnegative}. That is, $u$ is sign-changing.

In fact, we will prove a more general result. Consider the following subset of the Nehari manifold which contains all sign-changing solutions of \eqref{D}:
$$
\mathcal{M}_\lambda
:= 
\left\{
w \in W_0^{1,p}(\Omega):~ w^\pm \not\equiv 0,~
H_\lambda(w^\pm) = \int_\Omega f w^\pm \, dx
\right\}
\subset
\mathcal{N}_\lambda.
$$
Consider also the corresponding minimization problem:
$$
\beta := \inf\{E_\lambda(w):~ w \in \mathcal{M}_\lambda\}.
$$
\begin{prop}\label{prop:existence_of_sign-changing_solution}
	Let $p>1$, $f>0$ a.e.\ in $\Omega$, and $\lambda_f^* < \lambda_2$. If $\lambda \in (\lambda_f^*, \lambda_2)$, then $\mathcal{M}_\lambda \neq \emptyset$, $\beta \in (-\infty,0)$, and $\beta$ is attained. 
	Moreover, $u$ is a minimizer for $\beta$ if and only if $u$
	is a ground state solution of \eqref{D}. 
	In particular, $u$ is sign-changing and $E_\lambda(u) < 0$.
\end{prop}
\begin{proof}
	\textit{Step 1.} Let us fix $\lambda \in (\lambda_f^*,\lambda_2)$. 
	We start by showing that $\mathcal{M}_\lambda \neq \emptyset$ and $\beta<0$.
	Let $u$ be a minimizer for $\lambda_f^*$ given by Lemma \ref{lem:l<l<l} \ref{lem:l<l<l:0}. 
	In view of the assumption $\lambda_f^* < \lambda_2$, Lemma \ref{prop:existence_zero} implies that $u$ is a sign-changing solution of {\renewcommand\nlab{\lambda_f^*}\eqref{D}} with $E_{\lambda_f^*}(u)=0$. 
	Then Lemma \ref{lem:neh} \ref{lem:neh:1} yields $H_{\lambda_f^*}(u)=0$, and 
	hence we conclude that
	$$
	H_{\lambda_f^*}(u^+) = \int_\Omega f u^+ \, dx > 0,
	\quad 
	H_{\lambda_f^*}(u^-) = \int_\Omega f u^- \, dx < 0,
	$$
	thanks to the assumption $f>0$.
	The continuity of $H_{\lambda}$ with respect to $\lambda$ implies the existence of $\tilde{\lambda} \in (\lambda_f^*,+\infty)$ such that $H_{\lambda}(u^+) > 0$ for all $\lambda \in (\lambda_f^*,\tilde{\lambda})$, and $H_{\tilde{\lambda}}(u^+) = 0$. 
	Note that $\tilde{\lambda} \geq \lambda_2$. Indeed, if we suppose that $\tilde{\lambda} < \lambda_2$, then $H_{\tilde{\lambda}}(u^+) = 0$ and $H_{\tilde{\lambda}}(u^-) < 0$ lead to a contradiction thanks to Lemma \ref{lem:second_eigenvalue}.
	Consequently, recalling that $\lambda \in (\lambda_f^*, \lambda_2)$, we have
	\begin{equation}\label{eq:HHHH}
	H_{\lambda_f^*}(u^+) > H_{\lambda}(u^+) > 0,
	\quad 
	H_{\lambda}(u^-) < H_{\lambda_f^*}(u^-) < 0.
	\end{equation}
	Thus, Lemma \ref{lem:neh} \ref{lem:neh:3} yields the existence of $t_{+} > 0$ and $t_{-} > 0$ such that $t_{+} u^+, t_{-} u^- \in \mathcal{N}_\lambda$, and hence $t_{+} u^+ + t_{-} u^- \in \mathcal{M}_\lambda$. 
	Therefore, $\mathcal{M}_\lambda \neq \emptyset$. 
	Moreover, by \eqref{eq:tu} and \eqref{eq:HHHH},
	\begin{equation}\label{eq:E<0}
	E_\lambda(t_{+} u^+ + t_{-} u^-) =  
	\left(
	\frac{1}{p}-1
	\right) 
	\left(
	\frac{\left|\int_\Omega f u^+ \, dx\right|^\frac{p}{p-1}}{\left|H_\lambda(u^+)\right|^\frac{1}{p-1}} 
	-
	\frac{\left|\int_\Omega f u^- \, dx\right|^\frac{p}{p-1}}{\left|H_\lambda(u^-)\right|^\frac{1}{p-1}}
	\right)
	< E_{\lambda_f^*}(u) = 0,
	\end{equation}
	which shows that $\beta < 0$.
	
	\textit{Step 2.} Let us prove that $\beta > -\infty$. 
	The fact $\mathcal{M}_\lambda \neq \emptyset$ implies the existence of a minimizing sequence $\{v_n\} \subset \mathcal{M}_\lambda$ for $\beta$, and in view of \eqref{eq:E<0} we can assume that $E_\lambda(v_n) < 0$ for each $n \in \mathbb{N}$. 
	It is enough to show that $\{v_n\}$ is bounded in $W_0^{1,p}(\Omega)$. 
	
	Below, we will denote by $\{w_n^\pm\}$ the sequence of normalized functions $w_n^\pm:=\frac{v_n^\pm}{\|\nabla v_n^\pm\|}$, $n \in \mathbb{N}$. 
	Since $\|\nabla w_n^\pm\|=1$ for each $n \in \mathbb{N}$, $\{w_n^\pm\}$ converges (up to a subsequence) weakly in $W_0^{1,p}(\Omega)$, strongly in $L^p(\Omega)$, and almost everywhere in $\Omega$ to some $w^\pm \in W_0^{1,p}(\Omega)$. In particular, $w^+ \geq 0$ and $w^- \leq 0$. Moreover, $w^- \not\equiv 0$ since $H_\lambda(w_n^-) < 0$ for all $n \in \mathbb{N}$, and 
	\begin{equation}\label{eq:Hw-<0}
	H_\lambda(w^-) \leq \liminf_{n \to +\infty} H_\lambda(w_n^-) \leq 0.
	\end{equation}	
	
	First, we show that $\{v_n^+\}$ is bounded in $W_0^{1,p}(\Omega)$. 
	Suppose, by contradiction, that $\|\nabla v_n^+\| \to +\infty$ as $n \to +\infty$, up to a subsequence.
	Then, we have
	$$
	0 < H_\lambda(w_n^+) = \frac{1}{\|\nabla v_n^+\|^{p-1}}\int_\Omega f w_n^+ \, dx \leq \frac{C}{\|\nabla v_n^+\|^{p-1}} \to 0,
	$$	
	which implies that $\|w_n^+\| > c_1$ for some $c_1>0$ and all $n \in \mathbb{N}$. That is, $w^+ \not\equiv 0$, and 
	$$
	H_\lambda(w^+) \leq \liminf_{n \to +\infty} H_\lambda(w_n^+) = 0.
	$$
	Thus, recalling \eqref{eq:Hw-<0}, Lemma \ref{lem:second_eigenvalue} applied to $w$ gives a contradiction with $\lambda < \lambda_2$, which shows that $\{v_n^+\}$ is bounded in $W_0^{1,p}(\Omega)$.
	
	Second, we show that $\{v_n^-\}$ is bounded in $W_0^{1,p}(\Omega)$. 
	By Lemma \ref{lem:neh} \ref{lem:neh:1}, we have
	\begin{equation}\label{eq:0<E}
	0 < E_\lambda(v_n^-) = \left(\frac{1}{p}-1\right) \int_\Omega f v_n^- \, dx = 
	\left(\frac{1}{p}-1\right) \|\nabla v_n^-\| \int_\Omega f w_n^- \, dx.
	\end{equation}
	Suppose, by contradiction, that $\|\nabla v_n^-\| \to +\infty$ as $n \to +\infty$, up to a subsequence.
	Note that $E_\lambda(v_n^-) < c_2 < +\infty$ for some $c_2>0$ and all $n \in \mathbb{N}$, since each $E_\lambda(v_n) < 0$, and $\{v_n^+\}$ is bounded. 
	Therefore, we deduce from \eqref{eq:0<E} and Lebesgue's dominated convergence theorem that $\int_\Omega f w^- \, dx = 0$. However, this is impossible since $w^- \leq 0$, $w^- \not\equiv 0$, and $f>0$. 	
	Thus, $\{v_n^-\}$ is bounded in $W_0^{1,p}(\Omega)$. 
	
	Therefore, we conclude that $\{v_n\}$ is bounded in $W_0^{1,p}(\Omega)$, which yields $\beta > -\infty$. 
	
	\textit{Step 3.} Let us now prove that $\beta$ is attained. 
	We may assume that the minimizing sequence $\{v_n\}$ converges (up to a subsequence) weakly in $W_0^{1,p}(\Omega)$ and strongly in $L^p(\Omega)$ to some $v = v^+ + v^- \in W_0^{1,p}(\Omega)$, since $\{v_n\}$ is bounded in $W_0^{1,p}(\Omega)$. Our claim will follow if we show that $\{v_n\}$ converges (up to a subsequence) strongly in $W_0^{1,p}(\Omega)$ and $v \in \mathcal{M}_\lambda$. 
	
	First, we show that $v^+ \not\equiv 0$. 
	Suppose that $\|\nabla v_n^+\| \to 0$.
	Since $E_\lambda(v_n) < 0$, $E_\lambda(v_n^+) < 0$, and $E_\lambda(v_n^-) > 0$ for each $n \in \mathbb{N}$, we have
	$$
	E_\lambda(v_n) = E_\lambda(v_n^+) + E_\lambda(v_n^-) > E_\lambda(v_n^+) \to 0,
	$$	
	which contradicts the minimization property of $\{v_n\}$. 
	Thus, $\|\nabla v_n^+\| > c_3$ for some $c_3 > 0$ and all $n \in \mathbb{N}$.
	Suppose now that $\|v_n^+\| \to 0$. Then we have
	$$
	0 < \int_\Omega |\nabla v_n^+|^p \, dx - \lambda 
	\int_\Omega |v_n^+|^p \, dx
	= \int_\Omega f v_n^+ \, dx \leq C \|v_n^+\| \to 0
	$$
	which is impossible since $\|\nabla v_n^+\| > c_3 > 0$. 
	Therefore, $v^+ \not\equiv 0$.
	Moreover, 
	\begin{equation}\label{eq:Hv+>0}
	0 < H_\lambda(v^+) \leq \liminf_{n \to +\infty} H_\lambda(v_n^+) = \int_\Omega f v^+ \, dx.
	\end{equation}
	Indeed, if we suppose that $H_\lambda(v^+) \leq 0$, then, recalling \eqref{eq:Hw-<0}, we apply Lemma \ref{lem:second_eigenvalue} to $u = v^+ + w^-$ and again get a contradiction with $\lambda < \lambda_2$.
	
	Second, we show that $v^- \not\equiv 0$. 
	Suppose that $\|\nabla v_n^-\| \to 0$ as $n \to +\infty$. 
	In this case, the Nehari constraint $v_n^- \in \mathcal{N}_\lambda$ implies 
	$$
	\int_\Omega f w_n^- \, dx = \|\nabla v_n^-\|^{p-1} H_\lambda(w_n^-) \to 0
	\quad 
	\text{as }
	n \to +\infty, 
	$$
	and we get a contradiction with $w^- \not\equiv 0$. Therefore, there exists $c_4>0$ such that $\|\nabla v_n^-\| > c_4 > 0$ for all $n \in \mathbb{N}$.
	Then, $H_\lambda(v_n^-) < 0$ implies $v^- \not\equiv 0$, and hence
	\begin{equation}\label{eq:Hv-<0}
	H_\lambda(v^-) \leq \liminf_{n \to +\infty} H_\lambda(v_n^-) =  
	\int_\Omega f v^- \, dx < 0.
	\end{equation}
	
	Third, we show that $\{v_n\}$ converges (up to a subsequence) to $v=v^+ + v^-$ strongly in $W_0^{1,p}(\Omega)$. Suppose, by contradiction, that 
	\begin{equation}\label{eq:v<liminfvn}
	\|\nabla v\| < \liminf_{n \to +\infty} \|\nabla v_n\|.
	\end{equation}
	In view of \eqref{eq:Hv+>0} and \eqref{eq:Hv-<0}, Lemma \ref{lem:neh} \ref{lem:neh:3} implies the existence of $t_+, t_->0$ such that $t_+ v^+, t_- v^- \in \mathcal{N}_\lambda$ and hence $t_+ v^+ + t_- v^- \in \mathcal{M}_\lambda$. Moreover, $t_+ \neq 1$ or $t_- \neq 1$ due to \eqref{eq:v<liminfvn}.
	Since $t_+$ is a point of global minimum of $E_\lambda(t v^+)$ with respect to $t>0$, and $t=1$ is a point of global maximum of $E_\lambda(t v_n^-)$ with respect to $t>0$ for each $n \in \mathbb{N}$, we get
	\begin{align*}
	\beta = \inf_{\mathcal{M}_\lambda} E_\lambda 
	&\leq 
	E_\lambda(t_+ v^+ + t_- v^-) 
	=
	E_\lambda(t_+ v^+) + E_\lambda(t_- v^-)
	\leq
	E_\lambda(v^+) + E_\lambda(t_- v^-)\\
	&<
	\liminf_{n \to +\infty}E_\lambda(v_n^+)  
	+
	\liminf_{n \to +\infty}E_\lambda(t_- v_n^-) 
	\leq 
	\liminf_{n \to +\infty}E_\lambda(v_n^+)  +
	\liminf_{n \to +\infty}E_\lambda(v_n^-)  =
	\beta,
	\end{align*}
	a contradiction. 
	Thus, $\{v_n\}$ converges to $v$ strongly in $W_0^{1,p}(\Omega)$, up to a subsequence, which yields $v \in \mathcal{M}_\lambda$. That is, $\beta$ is achieved at $v$.
	
	\textit{Step 4.} Let us prove that $v$ is a critical point of $E_\lambda$.
	Note that the constraints for $v^\pm$ given by $\mathcal{M}_\lambda$ are not necessarily differentiable. That is, we cannot use the Lagrange multipliers rule or a deformation lemma over $\mathcal{M}_\lambda$. 
	To obtain the result, we employ the quantitative deformation lemma \cite[Lemma 2.3]{willem} over $W_0^{1,p}(\Omega)$. The following arguments are inspired by \cite{BWW}.
	
	Suppose, by contradiction, that  $\|E_\lambda'(v)\|_{(W_0^{1,p}(\Omega))^*} > 0$. 
	Since $E_\lambda \in C^1(W_0^{1,p}(\Omega), \mathbb{R})$, we can find 
	$\varepsilon,\delta > 0$ such that if $E_\lambda(w) \in [\beta-2{\varepsilon},\beta+2{\varepsilon}]$ and $\|\nabla(v-w)\| \leq 2\delta$ for some $w$, then $\|E_\lambda'(w)\|_{(W_0^{1,p}(\Omega))^*} \geq \frac{8\varepsilon}{\delta}$. 
	Then, using \cite[Lemma 2.3]{willem}, we obtain a continuous family of homeomorphisms $\Phi(\cdot,\tau)$ in $W_0^{1,p}(\Omega)$, $\tau \in [0,1]$, such that 
	\begin{enumerate}[label={\rm(\roman*)}]
		\item\label{prop:prop:existence_of_sign-changing_solution:1} $\Phi(w,\tau) = w$ provided $\tau=0$ or $\tau \in (0,1]$ and $|E_\lambda(w) -\beta| \geq 2\varepsilon$;
		\item\label{prop:prop:existence_of_sign-changing_solution:3} $E_\lambda(\Phi(w,\tau)) < \beta$ for any $\tau \in (0,1]$ provided $E_\lambda(w) \leq \beta$ and $\|\nabla(v-w)\| \leq \delta$.
	\end{enumerate}
	We will reach a contradiction by considering the deformations $\Phi(g(s),\tau)$ of the function $g(s) \in W_0^{1,p}(\Omega)$ defined as 
	$$
	g(s) := v^+ + s v^-, 
	\quad 
	s > 0.
	$$
	Note that $g(s)^+ = v^+$ and $g(s)^- = s v^-$. 
	Since $v \in \mathcal{M}_\lambda$, \eqref{eq:Hv-<0} and Lemma \ref{lem:neh} \ref{lem:neh:3} imply that  $s=1$ is a unique point of global maximum of $E_\lambda(s v^-)$ with respect to $s>0$, which yields
	\begin{equation}\label{eq:Egr<Ev}
	E_\lambda(g(s)) = 
	E_\lambda(v^+) + E_\lambda(s v^-) < E_\lambda(v^+) + E_\lambda(v^-) = \beta
	\end{equation}
	for any $s \neq 1$, and 
	\begin{align}
	\label{eq:Q<0}
	&\left.\frac{\partial}{\partial t} E_\lambda(t g(s)^-) \right|_{t=1} = H_\lambda(g(s)^-) - \int_\Omega f g(s)^- \,dx > 0
	\quad
	\text{for } s  \in (0,1),\\
	\label{eq:Q>0}
	&\left.\frac{\partial}{\partial t} E_\lambda(t g(s)^-) \right|_{t=1} =  H_\lambda(g(s)^-) - \int_\Omega f g(s)^- \,dx < 0
	\quad
	\text{for } s > 1.
	\end{align}
	
	Let us take some constants $\kappa_\pm > 0$ such that $\kappa_- < 1 < \kappa_+$ and $\|\nabla(v-g(s))\| \leq \delta$ for any $s \in [\kappa_-, \kappa_+]$. Considering $\varepsilon>0$ smaller, if necessary, we may assume by \eqref{eq:Egr<Ev} that
	\begin{equation}\label{eq:Ekappa<E}
	\max\left\{
	E_\lambda(g(\kappa_-)), E_\lambda(g(\kappa_+))
	\right\} < \beta - 2 \varepsilon.
	\end{equation}
	Thus, we deduce from \eqref{eq:Ekappa<E} and assertion \ref{prop:prop:existence_of_sign-changing_solution:1} that $\Phi(g(\kappa_\pm),\tau) = g(\kappa_\pm)$ for any $\tau \in [0,1]$.
	Therefore, we see from \eqref{eq:Q<0}, \eqref{eq:Q>0}, 
	and the continuity of $\Phi$, that for any $\tau \in [0,1]$ there exists $s_0 \in (\kappa_-,\kappa_+)$ such that
	\begin{equation}
	\label{eq:Hphi=0}
	H_\lambda(\Phi(g(s_0),\tau)^-) - \int_\Omega f \Phi(g(s_0),\tau)^- \,dx = 0.
	\end{equation}
	Moreover, since $g(s)^\pm \not\equiv 0$ for any $s>0$ and \eqref{eq:Hv+>0} is satisfied, 	
	assertion \ref{prop:prop:existence_of_sign-changing_solution:1} and the continuity of $\Phi$ imply the existence of sufficiently small $\tau_0>0$ such that $\Phi(g(s_0),\tau_0)^\pm \not\equiv 0$ and
	\begin{equation}
	\label{eq:h>0h<0}
	H_\lambda(\Phi(g(s_0),\tau_0)^+) > 0.
	\end{equation}
	In particular, \eqref{eq:Hphi=0} yields $\Phi(g(s_0),\tau_0)^- \in \mathcal{N}_\lambda$.  
	Furthermore, by \eqref{eq:h>0h<0} and Lemma \ref{lem:neh} \ref{lem:neh:3}, there exists $t_+ > 0$ such that
	\begin{align}
	\notag
	&H_\lambda(t_+\Phi(g(s_0),\tau_0)^+) - \int_\Omega f t_+ \Phi(g(s_0),\tau_0)^+ \,dx = 0,\\
	\label{eq:Et<E}
	&E_\lambda(t_+\Phi(g(s_0),\tau_0)^+) \leq 
	E_\lambda(\Phi(g(s_0),\tau_0)^+).
	\end{align}
	Therefore, we also have $t_+\Phi(g(s_0),\tau_0)^+ \in \mathcal{N}_\lambda$, which gives $t_+\Phi(g(s_0),\tau_0)^+ + \Phi(g(s_0),\tau_0)^- \in \mathcal{M}_\lambda$. 
	
	Since $E_\lambda(g(s)) = E_\lambda(\Phi(g(s),0)) \leq \beta$ for any $s>0$ by \eqref{eq:Egr<Ev}, and $s_0 \in (\kappa_-,\kappa_+)$, assertion \ref{prop:prop:existence_of_sign-changing_solution:3} and the choice of $\kappa_\pm$ imply
	\begin{equation}\label{eq:E<beta-1}
	E_\lambda(\Phi(g(s_0),\tau_0)) < \beta.
	\end{equation}
	Thus, using \eqref{eq:Et<E} and \eqref{eq:E<beta-1},
	we obtain the following contradiction:
	\begin{align*}
	\beta 
	= \inf_{\mathcal{M}_\lambda} E_\lambda 
	&\leq 
	E_\lambda(t_+\Phi(g(s_0),\tau_0)^+ + \Phi(g(s_0),\tau_0)^-) = E_\lambda(t_+\Phi(g(s_0),\tau_0)^+) + E_\lambda(\Phi(g(s_0),\tau_0)^-)\\
	&\leq
	E_\lambda(\Phi(g(s_0),\tau_0)^+) + E_\lambda(\Phi(g(s_0),\tau_0)^-) = 
	E_\lambda(\Phi(g(s_0),\tau_0)) < \beta.
	\end{align*}
	That is, $v$ is a critical point of $E_\lambda$. 
	
	\textit{Step 5.} To finish the proof, let us recall that $\{v_n\}$ was an arbitrary minimizing sequence for $\beta$, and $\beta<0$. That is, any minimizer $u$ for $\beta$ is a sign-changing solution of \eqref{D} with $E_\lambda(u)=\beta$. Consequently, any ground state solution $w$ of \eqref{D} satisfies $E_\lambda(w) \leq \beta$. 
	Note that $w$ is also sign-changing and hence $w \in \mathcal{M}_\lambda$. Indeed, Lemma \ref{lem:nonnegative} implies that $w$ is either nonpositive or sign-changing. However, if we suppose that $w$ is nonpositive, then $\int_\Omega f w \, dx \leq 0$, and hence $E_\lambda(w) \geq 0$ by Lemma \ref{lem:neh} \ref{lem:neh:1}, which contradicts $\beta<0$. 
	Therefore, we see that $w \in \mathcal{M}_\lambda$ and $E_\lambda(w)=\beta$, which establishes the desired claim that $u$ is a minimizer for $\beta$ if and only if $u$ is a ground state solution of \eqref{D}.
\end{proof}

Arguing in a similar (but simpler) way as in Proposition \ref{prop:existence_of_sign-changing_solution}, the following fact can be proved.
\begin{lemma}\label{rem:ground_state}
	Let $p>1$ and $\lambda \in (\lambda_1,\lambda_f^*)$. 	
	Then $u$ is a ground state solution of \eqref{D} if and only if $u$ is a minimizer for the problem
	$$
	\inf\left\{E_\lambda(w):~ w \in \mathcal{N}_\lambda, ~ E_\lambda(w) > 0\right\}.
	$$
\end{lemma}

Let us finally complete the proof of Theorem \ref{thm:1} \ref{thm:1:2}.
\begin{lemma}\label{lem:l<<l}
	Let $p>1$. If $\lambda_f^* < \lambda_2$, then $\lambda_f < \lambda_f^*$.
\end{lemma}
\begin{proof}
	We know from Corollary \ref{cor:l<l} that $\lambda_f \leq \lambda^*_f$.
	Suppose, by contradiction, that $\lambda_f = \lambda^*_f$. 
	Then for any sequence $\{\mu_n\} \subset (\lambda_1, \lambda_f^*)$ such that $\mu_n \to \lambda_f^*$ we can obtain from Lemma \ref{lem:ground_state} a sequence of corresponding ground state solutions $\{u_n\}$ of {\renewcommand\nlab{\mu_n}\eqref{D}} with $u_n < 0$ in $\Omega$.
	
	Let us normalize each $u_n$ as $v_n := \frac{u_n}{\|\nabla u_n\|}$.
	Lemma \ref{lem:H<0} implies $H_{\mu_n}(v_n) < 0$
	for every $n \in \mathbb{N}$, and hence $v_n$ converges (up to a subsequence) to some $v \leq 0$, $v \not\equiv 0$, weakly in $W_0^{1,p}(\Omega)$ and strongly in $L^p(\Omega)$. 
	Moreover, each $v_n$ satisfies the equation
	$$
	-\Delta_p v_n = \mu_n |v_n|^{p-2}v_n + \frac{f(x)}{\|\nabla u_n\|^{p-1}}
	\quad
	\text{in }
	\Omega,
	$$
	in the weak sense.
	Suppose first that $\{u_n\}$ is unbounded in $W_0^{1,p}(\Omega)$. 
	Then we see from the last equation that $v$ is an eigenfunction and $\lambda_f^*$ is a corresponding eigenvalue. 
	However, this is impossible since $\lambda_f^* \in (\lambda_1, \lambda_2)$. Therefore, $\{u_n\}$ is a bounded sequence in $W_0^{1,p}(\Omega)$, and hence $\{u_n\}$ converges (up to a subsequence) to some $u \leq 0$ weakly in $W_0^{1,p}(\Omega)$ and strongly in $L^p(\Omega)$.
	
	Suppose that $\|\nabla u_n\| \to 0$ as $n \to +\infty$. 
	Since each $u_n$ is a solution of {\renewcommand\nlab{\mu_n}\eqref{D}}, we have
	$$
	\int_\Omega f \xi \, dx = \left<H_{\mu_n}'(u_n), \xi\right>
	\to 0
	\quad
	\text{as }
	n \to +\infty
	\quad
	\text{for any }
	\xi \in W_0^{1,p}(\Omega),
	$$
	which yields $f \equiv 0$, a contradiction. 
	Therefore, there exists $c_1>0$ such that $\|\nabla u_n\| \geq c_1$ for all $n \in \mathbb{N}$.
	Since $H_{\mu_n}(u_n) < 0$ by Lemma \ref{lem:H<0}, there exists $c_2>0$ such that $\|u_n\| \geq c_2 > 0$ for all $n \in \mathbb{N}$. That is, $u \not\equiv 0$.
	
	We conclude from the weak convergence that $u$ is a solution of {\renewcommand\nlab{\lambda_f^*}\eqref{D}}, and by the weak lower semicontinuity argument we have $\liminf\limits_{n\to +\infty}E_{\mu_n}(u_n) \geq E_{\lambda_f^*}(u)$.
	Since $u \leq 0$, Lemma \ref{lem:ground_state_0} implies that $E_{\lambda_f^*}(u) > 0$.
	Let us show that this is impossible. 
	Recall that any minimizer $v$ for $\lambda_f^*$ (which exists by Lemma \ref{lem:l<l<l} \ref{lem:l<l<l:0}) is a sign-changing solution of 	
	{\renewcommand\nlab{\lambda_f^*}\eqref{D}} with  $E_{\lambda_f^*}(v)=0$, see Lemma \ref{prop:existence_zero}. 
	We have two possibilities: either $E_{\lambda_f^*}(v^\pm) = 0$ or $E_{\lambda_f^*}(v^\pm) \neq 0$. In the former case, $H_{\lambda_f^*}(v^\pm)=0$ by Lemma \ref{lem:l<l<l} \ref{lem:l<l<l:1}, which contradicts Lemma \ref{lem:second_eigenvalue}. Therefore, the latter case takes place, which yields $H_{\lambda_f^*}(v^\pm)\neq 0$. 
	By Lemma \ref{lem:neh} \ref{lem:neh:3} and the continuity of $H_\lambda$ with respect to $\lambda$, for any sufficiently large $n \in \mathbb{N}$ we can find $t_n^\pm>0$ such that $t_n^\pm v^\pm \in \mathcal{N}_{\mu_n}$, and $t_n^\pm \to 1$ as $n \to+\infty$. That is, $t_n^+ v^+ + t_n^- v^- \in \mathcal{M}_{\mu_n} \subset \mathcal{N}_{\mu_n}$, and $E_{\mu_n}(t_n^+ v^+ + t_n^- v^-) \to 0$ as $n \to +\infty$. Thus, we get a contradiction, since any $u_n$ is a minimizer of $E_{\mu_n}$ over $\mathcal{N}_{\mu_n}$ (see Lemma \ref{rem:ground_state}), but $\liminf\limits_{n\to +\infty}E_{\mu_n}(u_n) \geq E_{\lambda_f^*}(u) > 0$.
	The proof is complete.
\end{proof}

\section{Discussion}\label{sec:discuss}

Let us provide several final remarks.

\begin{enumerate}
	\item The critical value $\lambda_f^*$ defined by \eqref{eq:l*} can be obtained via a general theory developed in \cite{ilIzv,ilNeh}.
	\item Along with \ref{AMP}, one can consider a \textit{weak} anti-maximum principle which states that there exists $\widetilde{\lambda}_f \geq \lambda_f > \lambda_1$ such that any solution $u$ of \eqref{D} with $\lambda \in (\lambda_1, \widetilde{\lambda}_f)$ satisfies $u \leq 0$. In this case, the estimates of Theorem \ref{thm:1} remain valid for $\widetilde{\lambda}_f$ instead of $\lambda_f$.
	\item It was proved in \cite[Theorems 17 and 27]{arcoya} that \ref{AMP} holds true provided $f \in L^\infty(\Omega)$ satisfies a weaker assumption $\int_\Omega f \varphi_1 \, dx > 0$ instead of $f \geq 0$. 
	Although $\lambda_f^*$ is well-defined for such $f$, we do not know whether $\lambda_f^*$ bounds $\lambda_f$ as in Theorem \ref{thm:1}.
	\item Let us recall that the inequality $\lambda_f \leq \lambda_f^*$ is an open problem provided the following three assumptions are simultaneously satisfied: $p \neq 2$, $N\geq 2$, and $\lambda_f^* = \lambda_2$.
	\item Most of the existence results for \eqref{D}, as well as properties of $\lambda_f^*$, obtained in the present paper remain valid under considerably weaker assumptions on $f$ than $f \in L^\infty(\Omega) \setminus \{0\}$ and $f \geq 0$. However, since \ref{AMP} requiring the latter two assumptions was our primary object of study, we omitted general statements in order to keep the exposition more transparent.
\end{enumerate}

\smallskip
\bigskip
\noindent
\textbf{Acknowledgments.}
V. Bobkov and P. Dr\'abek were supported by the grant 18-03253S of the Grant Agency of the Czech Republic. V. Bobkov and Y. Ilyasov were also supported by the project LO1506 of the Czech Ministry of Education, Youth and Sports.
Y. Ilyasov wishes to thank the University of West Bohemia, where this research was started, for the invitation and hospitality.

\end{document}